\documentclass[a4paper,reqno,12pt]{amsart}

\usepackage{tgschola}   
\usepackage[T1]{fontenc}
\usepackage[utf8]{inputenc}
\usepackage{xspace,
		amsfonts}
\usepackage{amsmath}
\usepackage{amssymb}
\usepackage{graphicx,
		bbm,
		tikz,
		subfig}
  \usepackage{xfrac}
		\usepackage{dsfont}
\usepackage{amsthm}
\usepackage{mathtools}
\usepackage{enumitem}
\usepackage{verbatim}
\usepackage[autostyle]{csquotes}

\usepackage{pgfplots}
\pgfplotsset{compat=1.18}

\usepackage{cite}
\usepackage[normalem]{ulem}

\mathtoolsset{showonlyrefs}

\usepackage{geometry}
\newtheorem{theorem}{Theorem}[section]
\newtheorem{lemma}[theorem]{Lemma}
\newtheorem{proposition}[theorem]{Proposition}

\theoremstyle{remark}

\theoremstyle{definition}
\newtheorem{definition}[theorem]{Definition}

\newcommand{\deb}{\rightharpoonup}

\newcommand{\R}{\mathbb{R}}

\newcommand{\N}{\mathbb{N}}

\newcommand{\Eps}{\mathcal{E}}

\tikzstyle{nodino}=[circle,draw,fill,inner sep=0pt,minimum size=0.5mm]
\tikzstyle{infinito}=[circle,inner sep=0pt,minimum size=0mm]
\tikzstyle{nodo}=[circle,draw,fill,inner sep=0pt, minimum size=0.5*width("k")]
\tikzstyle{nodo_vuoto}=[circle,draw,inner sep=0pt, minimum size=0.5*width("k")]
\usetikzlibrary{graphs}
\tikzset{every loop/.style={min distance=10mm,in=300,out=240,looseness=10}}
\tikzset{place/.style={circle,thick,draw=blue!75,fill=blue!20,minimum
		size=6mm}}
\tikzset{place2/.style={circle,thick,draw=red!75,fill=red!20,minimum
		size=6mm}}

\title[NLSE with a $\delta'$ interaction]{Trapping by repulsion: \\the NLS with a delta-prime}

\author[R. Adami]{Riccardo Adami}
\address[R. Adami]{Politecnico di Torino, Dipartimento di Scienze Matematiche ``G.L. Lagrange'' Corso Duca degli Abruzzi 24, 10129, Torino, Italy.}
\email{riccardo.adami@polito.it}

\author[F. Boni]{Filippo Boni}
\address[F. Boni]{Scuola Superiore Meridionale, Largo S. Marcellino, 10, 80138, Napoli, Italy.}
\email{f.boni@ssmeridionale.it}

\author[M.~Gallone]{Matteo Gallone}
\address[M.~Gallone]{Scuola Internazionale Superiore di Studi Avanzati (SISSA), via Bonomea, 265, 34136 Trieste Italy}
\email{matteo.gallone@sissa.it}

\begin{document}

	\begin{abstract}
We establish the existence and provide explicit expressions for the stationary states of the one-dimensional Schrödinger equation with a repulsive delta-prime potential and a focusing nonlinearity of power type. Furthermore, we prove that, if the nonlinearity is subcritical, then 
ground states exist for any strength of the delta-prime interaction and for every positive value of the mass.

This result supplies an example of ground states arising from a repulsive potential, a counterintuitive phenomenon explained by the emergence of an additional dimension in the energy space, induced by the delta-prime interaction. This new dimension contains states of lower energy and is thus responsible for the existence of nonlinear ground states that do not originate from linear eigenfunctions.

The explicit form of the ground states is derived by addressing the ancillary problem of minimizing the action functional on the Nehari manifold. We solve such problem in the subcritical, critical, and supercritical regimes.

	\end{abstract}
	
	\maketitle

	\vspace{-.5cm}
	

\section{Introduction}\label{sec:Introduction}

Classical results show that the nonlinear Schrödinger equation (NLS) on the real line with a focusing, subcritical nonlinearity has a ground state for every value of the mass \cite{zakharov-shabat-78,caze-lions-82,cazenave-03}. A ground state for the NLS is defined as the function that minimizes the energy among all functions with the same mass, i.e.~the same $L^2$-norm.
 Up to translations and multiplications by a constant phase, the ground state at a given mass is unique, always positive, and does not change shape over time: only its phase evolves with a constant frequency. Such solutions are often called solitons.

If one adds a repulsive potential, e.g.~a  positive potential that vanishes at infinity, it is generally expected that these ground states no longer exist. This is because the added potential increases the energy of any function, making it impossible to reach the energy of the soliton in the absence of the potential. However, by translating the soliton towards infinity, such an energy can be approached without ever being reached. As a consequence, the infimum of the energy is not attained and there is no ground state.

In this paper, we present a counterexample to such a picture. To this end, we consider a  highly singular repulsive potential known as a $\delta'$
interaction, and place it at the origin of the line. The $\delta'$ potential belongs to the class of the point interactions as it has no effect to functions
that vanish around the origin. It can
be considered repulsive in two respects: first, its contribution to the energy is positive; second, in the absence of a nonlinear term, 
it does not sustain any bound state.

Remarkably, in presence of a nonlinear term we show not only that  ground states  exist, but also that they have a lower energy than the unperturbed soliton with the same mass.

Before going into the technical details, we briefly explain the  context to which our model 
belongs. 

Nonlinear Schr\"odinger equations  (NLS) with a point interaction have attracted in the last years a growing interest in  the physical as well as in the mathematical research. Physically, such equations are widely used to describe the dynamics of classical oscillator chains \cite{Bambusi2002,Gallone2021}, to model small-amplitude water waves \cite{Zakharov1972}, or  to recover several properties of the quantum many-body systems, such as the collective behaviour of the particles in a Bose-Einstein condensate \cite{KayKirkpatrick2011,JLP-2016,Li2021}. More generally, by their extremely localized nature, point interactions have been used to model short-range phenomena such as the effects of impurities or defects in a medium. In  particular, the one-dimensional $\delta'$ potential
has primarily been employed in the study of the Wannier–Stark effect (see \cite{Avron1994}).

    The equation we analyse  can be formally written as 
    \begin{equation}\label{eq:NLS-deltaPrime-Formal}
         i \partial_t \psi = \text{``}(-\Delta + \beta \delta')\text{''}\psi - |\psi|^{p-2}\psi \, , \qquad 2 < p
         < 6\, , \beta > 0,
    \end{equation}
where the quotation marks stress the formal character of the expression $- \Delta + \beta \delta'$.

Indeed, point interactions in space dimension $n$ are rigorously constructed as self-adjoint extensions of the Laplacian  restricted to 
the space $ C^\infty_0(\mathbb{R}^n \setminus \{0\})$ made of 
smooth functions whose support is compact and does not contain the origin \cite{Albeverio1988,Gallone-Michelangeli-Book}. 

In  dimension one, such a restricted Laplacian admits a four-parameter family of self-adjoint extensions, each corresponding by definition to a specific point interaction. In \eqref{eq:NLS-deltaPrime-Formal}
the symbols in quotation marks represent a 
particular choice in such a family.

The best known one-dimensional nonlinear model with a point interaction involves the Dirac's $\delta$ potential instead of the $\delta'$. NLS with $\delta$ interaction has been widely studied  on the real line \cite{Caudrelier2005,Cao-Malomed-95,Fuk-Ohta-Ozawa-08,Holm-Marz-Zwo-07,Fukuizumi2008-DCDS,HGoodman2004}, on the half-line \cite{Boni2023}, and on metric graphs \cite{Adami2012,Adami-Graph2014,Adami-graph2016,Serra2016,Boni-Dovetta-2022}.  It has been proved that in the presence of  an attractive $\delta$ interaction, nonlinear ground states exist for every value of the mass. On the contrary, if the $\delta$ interaction is repulsive, then ground states fail to exist for every value of the mass. In this regard, the $\delta$ interaction in dimension one exhibits the expected behaviour of an
    ordinary potential.

Concerning the  $\delta'$ interaction, in the attractive case all ground states  have been
singled out \cite{Adami-Noja-CMP2013-1D-NLS-deltaprime,Adami-Noja-Visc-13}. For the repulsive case the analysis of the stability was carried out for the sign-changing
solutions on the line \cite{Angulo-Goloshchapova-2020}, and for sign-preserving solutions on
star graphs \cite{Goloshchapova-2022}, in both cases leaving the ground states unconsidered.

More recently, non-linear dynamics with a different type of point interaction, the so-called F\"ul\"op-Tsutsui interaction, has been investigated too \cite{Adami-Rev,Flp-00,ABNR-25}.

In contrast with the one-dimensional case, in two and three dimensions point interactions always lead to ground states, regardless of whether the interaction is attractive or repulsive. In two dimensions, this fact is natural because at the linear level a point interaction always produces a negative eigenvalue. Adding a nonlinearity produces a branch of nonlinear ground states that bifurcate from the linear one.

In three dimensions the result is more surprising: even a repulsive delta interaction, which does not support bound states in the linear case, can give rise to nonlinear ground states. This happens as a consequence of a highly non trivial interplay between the singularity induced by the point interaction and nonlinearity. Mathematically, the singularity introduces
an additional dimension in the energy space, that acts as 
a sort of extra room allowing for states with a lower energy.

A similar but more subtle phenomenon happens for a  $\delta'$
interaction in one dimension: although the singularity here is a jump rather than a divergence, it still contributes negatively to the nonlinear energy. This peculiar phenomenon can be understood by recalling that a repulsive $\delta'$ potential can be
approximated by smooth potentials with a negative well \cite{Cheon1998,Exner2001}. We conjecture that the nonlinear ground state is inherited by the corresponding ground state induced by the approximating potentials. We plan to further investigate such a problem in the future.

As a matter of fact, for the subcritical nonlinearity $p < 6$
the existence of a ground state of \eqref{eq:NLS-delta'}
at every mass
is immediate, since the energy of a soliton centred at the origin is not affected by the presence of the
$\delta'$ (for more details see Sec.~\ref{sec:results}).
The energy of the soliton is then attained and, using concentration-compactness theory
as in \cite{Adami-Serra-Tilli-2016}, a ground state must exist.
However, rather than relying on abstract arguments, we explicitly construct ground states and show that their energy is strictly lower than that of the soliton.
This fact is particularly clear if the mass is large. Indeed, if one focuses on functions that vanish on the negative halfline, then the problem reduces to a NLS with a repulsive 
$\delta$
 interaction on the halfline,
 for which ground states have been proved to exist if the mass is large enough \cite{Boni2023}. A more detailed analysis is needed for small masses.

 Our main result is thus the existence of a ground state of \eqref{eq:NLS-delta'}
 at every mass in the presence of a subcritical nonlinearity $(2 < p < 6)$ and a repulsive 
 $\delta'$ interaction $(\beta > 0).$

 In order to identify the ground states it is useful to make recourse to the 
 easier problem of the minimization of the action functional, as ground states at some mass are
 also action minimizers at some frequency (see Lemma \ref{lem:link-GS-AM}, as well as \cite{DST-action-energy,Jeanjean-22}). Owing to this ancillary problem, one can
 explicitly represent all stationary states not only in the subcritical case, but even
in  the critical $(p=6)$ and in the supercritical $(p > 6)$. Whereas in the last case  there are no ground states, the critical case is usually interesting in itself and
 will be the subject of a forthcoming paper.

\begin{figure}[h!]
    \begin{tikzpicture}
  \begin{axis}[
    axis lines = middle,
    ylabel = {$\phi_\omega$},
    samples = 100,
    domain = -2:2,
    enlargelimits,
    xtick=\empty,
    ytick=\empty,
    xscale=1,
    yscale=0.66,
    ylabel style={anchor=west},
  ]
  
     \addplot[black, thick, domain=-1.5:0] {3/((cosh(3* x))^2)};
    
    \addplot[black, thick, domain=0:1.5] {3/((cosh(3* x))^2)};
    
    \node at (rel axis cs:0.95,0.95) {(a)};
  \end{axis}
\end{tikzpicture}

\begin{tikzpicture}
  \begin{axis}[
    axis lines = middle,
    ylabel = {$u_{1,\omega}$},
    samples = 100,
    domain = -2:2,
    enlargelimits,
    xtick=\empty,
    ytick=\empty,
    xscale=1,
    yscale=0.66,
    ylabel style={anchor=west},
  ]
  
     \addplot[black, thick, domain=-1.5:0] {3/((cosh(3* x+0.3))^2)};
    
    \addplot[black, thick, domain=0:1.5] {3/((cosh(3* x+0.9))^2)};
    
    \node at (rel axis cs:0.95,0.95) {(b)};
  \end{axis}
\end{tikzpicture} 
\begin{tikzpicture}
  \begin{axis}[
    axis lines = middle,
    ylabel = {$\tilde{u}_{1,\omega}$},
    samples = 100,
    domain = -2:2,
    enlargelimits,
    xtick=\empty,
    ytick=\empty,
    xscale=1,
    yscale=0.66,
    ylabel style={anchor=west},
  ]
  
     \addplot[black, thick, domain=-1.5:0] {3/((cosh(3* x+0.9))^2)};
    
    \addplot[black, thick, domain=0:1.5] {3/((cosh(3* x+0.3))^2)};
    
    \node at (rel axis cs:0.95,0.95) {(c)};
  \end{axis}
\end{tikzpicture}
\begin{tikzpicture}
  \begin{axis}[
    axis lines = middle,
    ylabel = {$u_{2,\omega}$},
    samples = 100,
    domain = -2:2,
    enlargelimits,
    xtick=\empty,
    ytick=\empty,
    xscale=1,
    yscale=0.66,
    ylabel style={anchor=east},
  ]
  
     \addplot[black, thick, domain=-1.5:0] {-3/((cosh(3* x+0.9))^2)};
    
    \addplot[black, thick, domain=0:1.5] {3/((cosh(3* x-0.9))^2)};
    
    \node at (rel axis cs:0.95,0.95) {(d)};
  \end{axis}
\end{tikzpicture} 
\begin{tikzpicture}
  \begin{axis}[
    axis lines = middle,
    ylabel = {$\tilde{u}_{2,\omega}$},
    samples = 100,
    domain = -2:2,
    enlargelimits,
    xtick=\empty,
    ytick=\empty,
    xscale=1,
    yscale=0.66,   
    ylabel style={anchor=east},        
  ]
  
     \addplot[black, thick, domain=-1.5:0] {-3/((cosh(3* x-0.9))^2)};
    
    \addplot[black, thick, domain=0:1.5] {3/((cosh(3* x+0.9))^2)};
    
    \node at (rel axis cs:0.95,0.95) {(e)};     
  \end{axis}
\end{tikzpicture}
\begin{tikzpicture}
  \begin{axis}[
    axis lines = middle,
    ylabel = {$u_{3,\omega}$},
    samples = 100,
    domain = -2:2,
    enlargelimits,
    xtick=\empty,
    ytick=\empty,
    xscale=1,
    yscale=0.66,
    ylabel style={anchor=east},   
  ]
  
     \addplot[black, thick, domain=-1.5:0] {-3/((cosh(3* x+0.9))^2)};
    
    \addplot[black, thick, domain=0:1.5] {3/((cosh(3* x-0.3))^2)};
    
    \node at (rel axis cs:0.95,0.95) {(f)};    
  \end{axis}
\end{tikzpicture}
\begin{tikzpicture}
  \begin{axis}[
    axis lines = middle,
    ylabel = {$\tilde{u}_{3,\omega}$},
    samples = 100,
    domain = -2:2,
    enlargelimits,
    xtick=\empty,
    ytick=\empty,
    xscale=1,
    yscale=0.66,   
    ylabel style={anchor=east},
  ]
  
     \addplot[black, thick,  domain=-1.5:0] {-3/((cosh(3* x-0.9))^2)};
    
    \addplot[black, thick,  domain=0:1.5] {3/((cosh(3* x+0.3))^2)};
    
    \node at (rel axis cs:0.95,0.95) {(g)};     
  \end{axis}
\end{tikzpicture}
\caption{Standing waves for the NLS with $\delta'$ potential of strength $\beta$. (a), (b), (d) and (f) are the shapes of standing waves for $\beta>0$. (a), (c), (e), (g) are the shapes of standing waves for $\beta < 0$.}\label{fig:stationary}
\end{figure}
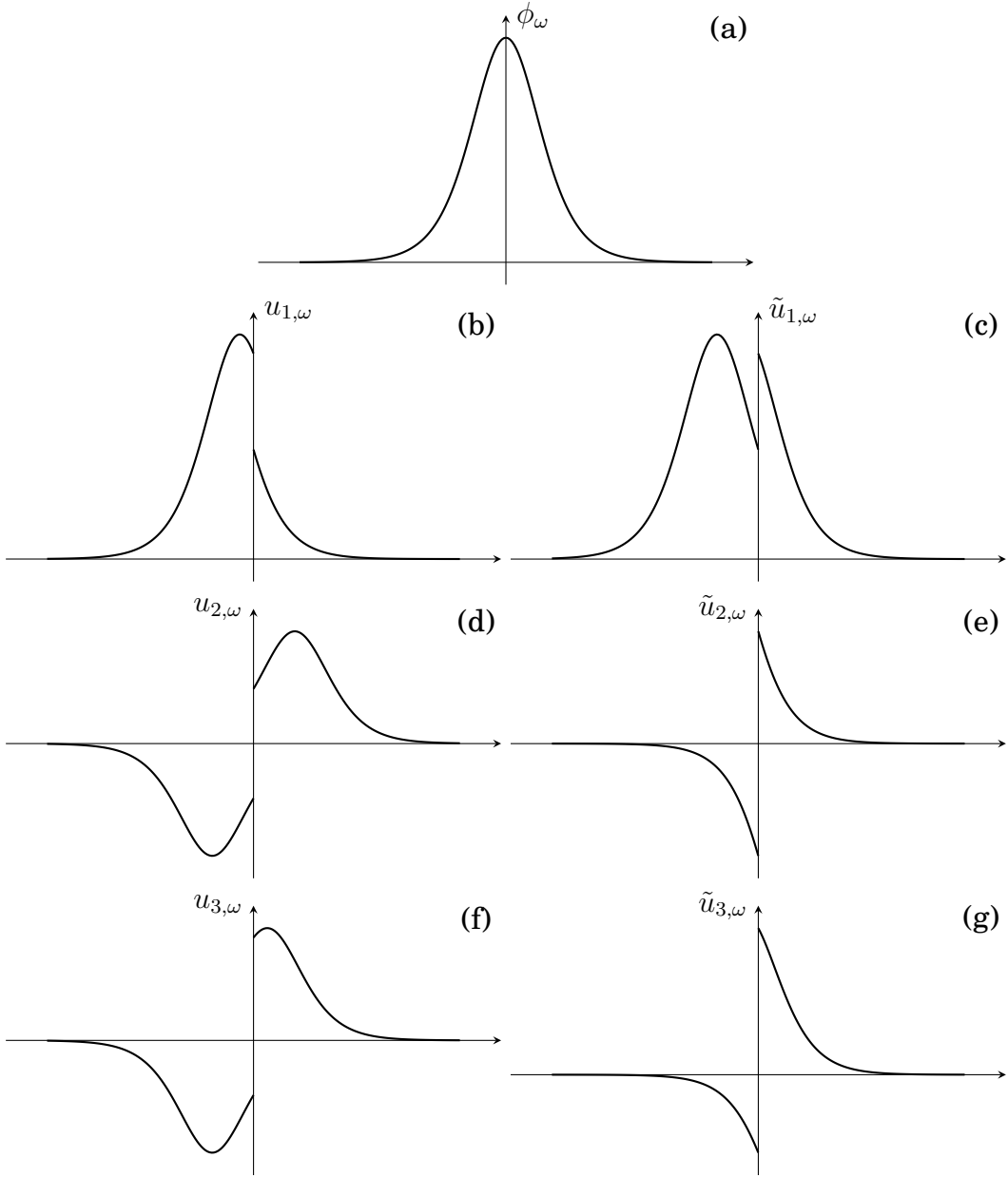

The complete family of stationary states, including the case $\beta < 0$ investigated in
\cite{Adami-Noja-Visc-13,Adami-Noja-CMP2013-1D-NLS-deltaprime,Angulo-Goloshchapova-2020,Goloshchapova-2022},
is represented in Fig.~\ref{fig:stationary}, and
can be illustrated as follows:
\begin{enumerate}
    \item Solitons centred at the origin, unaffected by the $\delta'$ potential (first row in 
Fig.~\ref{fig:stationary}). They are present for every positive frequency.
\item Positive stationary states bifurcating from the soliton as $\beta$ becomes non-zero
(second row in 
Fig.~\ref{fig:stationary}: (b) for $\beta > 0$, (c) for $\beta < 0$). They are present for every positive frequency.
\item Symmetric sign-changing solutions (third row in 
Fig.~\ref{fig:stationary}: (d) for $\beta > 0$, (e) for $\beta < 0$). They are present above
a first frequency threshold. Their branch bifurcates from the linear ground states 
as the nonlinearity appears. Indeed, the frequency threshold corresponds to the 
energy of the linear ground state.

  \item Non-symmetric sign-changing solutions (fourth row in 
Fig.~\ref{fig:stationary}: (f) for $\beta > 0$, (g) for $\beta < 0$). They are present above
a second frequency threshold, higher than the first.  Their branch bifurcates from  
(e) or (f) as the frequency crosses the second threshold.
\end{enumerate} 
Our main result establishes that ground states for $\beta > 0$ always belong to the branch (b).
It is already known that for $\beta < 0$ ground states belong to the branch (g), for any frequency above the second threshold, and to (e) for any frequency between the two thresholds 
\cite{Adami-Noja-Visc-13,Adami-Noja-CMP2013-1D-NLS-deltaprime}.

In \cite{Adami-Noja-Visc-13,Adami-Noja-CMP2013-1D-NLS-deltaprime} the bifurcations at point (3) and (4) are rigorously proven. The full bifurcation structure of all solutions remains partly conjectural and needs further study.

The paper is organized as follows:
\begin{itemize}
    \item In Sec.~\ref{sec:results} we give the mathematical formulation of the problem and
    the precise statement of the results;
    \item In Sec.~\ref{sec:stationary} we explicitly find all stationary states;
    \item In Sec.~\ref{sec:action} we identify and prove the existence of the action minimizers;
    \item In Sec.~\ref{sec:minchiazzio} we identify and prove the existence of the ground states.
\end{itemize}
The well-posedness of the time evolution problem has been treated in \cite{Goloshchapova-2022},
so we do not deal with it.

	\section{Setting and Main Results}\label{sec:results}
        The rigorous  formulation of equation \eqref{eq:NLS-deltaPrime-Formal} reads
        \begin{equation}
        \label{eq:NLS-delta'}
            i \partial_t \psi \ = \ H_\beta \psi - |\psi|^{p-2} \psi, \quad \beta>0,\quad p>2,
        \end{equation}
        where $H_\beta$ is the self-adjoint realisation of the 
        one-dimensional Laplacian with a
        repulsive $\delta'$ interaction of strength $\beta$ at the origin. More precisely, $H_\beta$ is a self-adjoint operator with domain and action given  by
\begin{equation*}
    D(H_\beta):=\left\{\psi\in H^2(\R\setminus\{0\})\,:\, \psi'(0^+)=\psi'(0^-)\,,\, \psi(0^+)-\psi(0^-)=\beta \psi'(0^+)\right\}
\end{equation*}
and 
\begin{equation*}
    H_\beta \psi=-\psi'',\quad x\neq 0.
\end{equation*}
For $\beta \neq 0$, the associated quadratic form $Q_\beta$ acts on the space $H^1(\R^-)\oplus H^1(\R^+)$ as
\begin{equation*}
Q_\beta(u):=\|u'\|_{L^2(\R^-)}^2+\|u'\|_{L^2(\R^+)}^2+\frac{1}{\beta}|u(0^+)-u(0^-)|^2,
\end{equation*}
while for $\beta=0$ the quadratic form $Q_0$ acts on the space $H^1(\R)$ as $Q_0(u):=\|u'\|_{L^2(\R)}^2$. Notice
that the domain of $Q_0$ is
strictly contained in that of $Q_\beta$ with $\beta \neq 0$.

Furthermore, the operator $H_\beta$ for $\beta>0$ has a purely continuous spectrum $[0,+\infty)$ only and no negative eigenvalues. 

We are interested in the standing waves, that is solutions of \eqref{eq:NLS-delta'} of the form 
\begin{equation} \label{standing}
    \psi(t,x)=e^{i\omega t}u(x)
    ,
    \end{equation}
    thus we reduce to search for solutions to the system
\begin{equation}
\label{eq:stationary}
\begin{cases}
    -u''+\omega u-|u|^{p-2}u=0\quad \text{on}\quad \R\setminus\{0\},\\
    u'(0^+)=u'(0^-),\\
    u(0^+)-u(0^-)=\beta u'(0^+).
    \end{cases}
\end{equation}
As widely known, there are two families 
of standing waves, characterized as the solutions to  two distinct minimization problems. The first one consists of the
ground states, that is the minimizers of the
energy among functions with the same mass. The
second family is made of the action minimizers, i.e.~the minimizers of the action
functional among all functions belonging to the Nehari, or natural, manifold. In fact, there is a Nehari manifold for every choice of
the frequency $\omega$ of oscillation of the standing wave (see \eqref{standing}), so one can interpret the action minimizers as the minimizers of the action at a given frequency. The precise relationship between the two notions was clarified in \cite{DST-action-energy,Jeanjean-22}.

In several studies the terminology of ground states is applied to the action minimizers. We prefer to distinguish the two notions,  and give more emphasis to the minimizers of the energy, that are physically meaningful.

In order to introduce the first variational setting,  fix $\beta \neq 0$
and
let $E_\beta:H^1(\R^-)\oplus H^1(\R^+)\to \R$ be the energy functional associated to \eqref{eq:NLS-delta'}, defined as 
\begin{equation*}
E_\beta(u)=\frac{1}{2}Q_\beta(u)-\frac{1}{p}\|u\|_{L^{p}(\R)}^{p}.
\end{equation*}
We introduce the  function $\Eps_\beta : \R^+ \to \R$ defined as
\begin{equation}
    \label{emu}
   \Eps_\beta(\mu):=\inf\left\{E_\beta(v)\,:\, v\in H^1(\R^-)\oplus H^1(\R^+)\,,\,\|v\|_{L^2(\R)}^2=\mu\right\}. 
\end{equation}
We can now introduce the fundamental notion of ground state.
\begin{definition}[Ground state]
We say that $u\in H^1(\R^-)\oplus H^1(\R^+)$ is a  {\em ground state at mass $\mu$} if \begin{equation}
\|u\|_{L^2(\R)}^2=\mu
\label{massconstraint}
\end{equation}
and
\begin{equation*}
E_\beta(u)=\Eps_\beta(\mu).
\end{equation*}
\end{definition}

The identity \eqref{massconstraint} is called the {\em mass constraint}.

To establish the second variational setting, let $\omega>0$ and  $S_\omega:H^1(\R^-)\oplus H^1(\R^+)\to \R$ be the 
{\em action functional} defined as
\begin{equation*}
S_{\beta,\omega}(u):=E_\beta(u)+\frac{\omega}{2}\|u\|_{L^2(\R)}^2.
\end{equation*}
We introduce the  Nehari manifold
\begin{equation*}
N_{\beta,\omega}:=\left\{u\in H^1(\R^-)\oplus H^1(\R^+)\setminus\{0\}\,:\,I_{\beta,\omega}(u)=0\right\},
\end{equation*}
where 
\begin{equation*}
    I_{\beta,\omega}(u)=\langle S'_{\beta,\omega}(u), u \rangle=Q_\beta(u)+\omega\|u\|_{L^2(\R)}^2-\|u\|_{L^p(\R)}^p.
\end{equation*}
We define the function $d_\beta : \R^+ \to \R$ as
\begin{equation}
    \label{dibbeta}
    d_\beta(\omega):=\inf\{S_{\beta,\omega}(v)\,:\, v\in N_{\beta,\omega}\}.
\end{equation}
We can now introduce the following variational problem:
\begin{definition}[Action minimizer]
We say that $u$ is an {\em action minimizer at frequency $\omega$}
if
\begin{equation*}
u\in N_{\beta,\omega}
\end{equation*}
and
\begin{equation*}
S_{\beta,\omega}(u)=d_\beta(\omega):=\inf\{S_{\beta,\omega}(v)\,:\, v\in N_{\beta,\omega}\}.
\end{equation*}
\end{definition}

In the following, a crucial role is played by the {\em soliton}, that is the solution to \eqref{eq:stationary} with $\beta=0$,
given by
\begin{equation}\label{eq:SolitonExplicit}
    \phi_\omega(x) \,:=\, \left( \frac{\omega \, p}{2 \cosh^2\big(\frac{p-2}{2} \sqrt{\omega} \,x\big)} \right)^{\frac{1}{p-2}} \, .
\end{equation}
Abusing notation, we denoted by $\phi_\omega$ the soliton at fixed frequency $\omega$ and shall denote by
$\phi_\mu$ the soliton at fixed mass, namely \eqref{eq:SolitonExplicit} with $\omega$ chosen such that $\Vert \phi_\omega \Vert_{L^2(\mathbb{R})}^2=\mu$. 
By a direct computation it
can be explicitly seen that the correspondence 
between frequency and mass is one-to-one,
differentiable, and monotonically strictly increasing (see Sec. \ref{sec:stationary}).

Before stating the first result, we recall that, if $\beta = 0$ and
$2<p<6$, then ground states  exist at every value of the mass and are given, up to symmetries, by  the soliton $\phi_\mu$, thus $E_0(\phi_\mu)=\mathcal{E}_0(\mu)$.

\begin{theorem}
\label{thm:ex-GS}
    Let $\beta>0$, $2<p<6$ and $\mu>0$. Then
    \begin{equation*}
\Eps_\beta(\mu)<E_0(\phi_\mu),
    \end{equation*}
    where $\phi_\mu$ is the soliton at mass $\mu$, and there exists a unique ground state $u$ at mass $\mu$, up to multiplication by a phase factor and to the transformation $u\mapsto u(-\cdot)$. Moreover, up to these symmetries, $u$ is positive, solves for some $\omega$ the system \eqref{eq:stationary} and, denoted with $\phi_\omega$ the soliton at frequency $\omega$, has the form
    \begin{equation}
\label{eq:profile-u-sol}
u(x)=
    \begin{cases}
    \phi_\omega(x+x_+),\quad x>0\\
        \phi_\omega(x+x_-),\quad x<0,
    \end{cases}
\end{equation}
with $x_+>x_->0$.

The quantities $t_\pm:=\tanh\left(\frac{p-2}{2} \sqrt{\omega}x_\pm\right)$ solve the system
\begin{equation}
\label{eq:t+t-pos}
    \begin{cases}
        t_+(1-t_+^2)^{\frac{1}{p-2}}=t_-(1-t_-^2)^{\frac{1}{p-2}},\\
        t_+^{-1}-t_-^{-1}=-\beta \sqrt{\omega}.
    \end{cases}
\end{equation} 
\end{theorem}

As highlighted in Sec.~\ref{sec:Introduction}, this result is somewhat unexpected since nonlinear ground states exist for every value of the mass even though  no linear eigenvalue 
exists. Indeed, by exploiting concentration-compactness principle, ground states exist if there exists a competitor with energy less or equal than the free energy of the soliton, but the soliton itself is a competitor and its energy $E_\beta$ coincide with its free energy, hence ground states exist. Furthermore, the energy level is strictly lower than the energy of the soliton and is realized by a positive function with a jump at the origin. The proof of this fact is based on the connection between ground states at fixed mass and action minimizers at fixed frequency (see Lemma \ref{lem:link-GS-AM}) and takes advantage of the fact that the comparison between the energies of different competitors is easier in the context of the minimization of the action functional at fixed frequency than in the context of the minimization of the energy functional at fixed mass. 

In the second theorem, we prove the existence of action minimizers at frequency $\omega$ if and only if  $\omega>0$ and we identify  the specific profile of such minimizers.

\begin{theorem}
\label{thm:ex-AM}
Let $\beta>0$ and $p>2$. Then there exists an action minimizer at frequency $\omega$ if and only if $\omega>0$. Moreover, up to a multiplication by a phase factor and the transformation $u\mapsto u(-\cdot)$, the action minimizer is unique, positive and is of the form \eqref{eq:profile-u-sol}-\eqref{eq:t+t-pos}.
\end{theorem}

Let us notice that, as one may expect, the result concerning action minimizers at fixed frequency is valid for any value of $p>2$, namely also in the $L^2$-critical and supercritical cases.

\section{Stationary states}\label{sec:stationary}
Preliminary to the proof of Theorems \ref{thm:ex-AM} and \ref{thm:ex-GS}, we analyze the stationary states, namely the solutions to \eqref{eq:stationary} belonging to $H^1(\R^-)\oplus H^1(\R^+)$. 

As already stated, if $\beta=0$, then the problem reduces to finding solutions in $H^1(\R)$ of the equation
\begin{equation}
\label{eq:statsoliton}
-u''+\omega u-|u|^{p-2}u=0.
\end{equation}
The only positive solutions in $H^1(\R)$ are given by the soliton $\phi_\omega$ in \eqref{eq:SolitonExplicit} and its translated. In particular, it is possible to compute explicitly the mass $\mu(\omega)$ of the soliton  $\phi_\omega$, which reads as
\begin{equation}
\label{eq:mu-omega}
\mu(\omega)=\|\phi_\omega\|_{L^2(\R)}^2=\frac{p^{\frac{2}{p-2}}}{(p-2)2^{\frac{4-p}{p-2}}}\omega^{\frac{6-p}{2(p-2)}}\int_{-1}^1 (1-t^2)^{\frac{4-p}{p-2}}\,dt,
\end{equation}
so that for $2<p<6$ the function $\omega\mapsto \mu(\omega)$ maps $(0,+\infty)$ into $(0,+\infty)$ and is bijective. One can also compute the inverse function $\omega(\mu)$, namely
\begin{equation}
\label{eq:omega-mu}
    \omega(\mu)=\left(\frac{(p-2)2^{\frac{4-p}{p-2}}\int_{-1}^1(1-t^2)^{\frac{4-p}{p-2}}\,dt}{p^\frac{2}{p-2}}\right)^\frac{2(p-2)}{6-p}\mu^{\frac{2(p-2)}{6-p}}.
\end{equation}
Therefore, one can alternatively refer to the soliton at fixed frequency $\omega$ or at fixed mass $\mu$, depending on the context. 

If $\beta>0$, then the set of stationary solutions is richer.

First of all, let us observe that, except from $u\equiv 0$, every other solution of the system \eqref{eq:stationary} does not vanish at any point $x\neq 0$ and, being continuous both on $\R^-$ and on $\R^+$, it does not change sign on each half-line: the change of sign can occur at the origin only, where the function presents a jump discontinuity. 

Moreover, this system presents some important symmetries. In fact, if $u$ is a solution, then $-u$ is a solution itself, so that without loss of generality we can assume that $u>0$ on $\R^+$. As a consequence, we can divide the set of solutions into two classes: the first class includes positive solutions on the whole real line ((a) and (b) in Fig.~\ref{fig:stationary}), while the second class includes solutions that are negative on $\R^-$ and positive on $\R^+$ ((d) and (f) in Fig.~\ref{fig:stationary}).

Let us start analyzing the first class, i.e.~the set of positive solutions of the system \eqref{eq:stationary}. 

\begin{lemma}
\label{lem:sol-pos}
Let $\beta>0$. If $\omega>0$, then the system \eqref{eq:stationary} admits three positive solutions belonging to $H^1(\R^-)\oplus H^1(\R^+)$. In particular, one of them is the soliton $\phi_\omega$ centered at $0$, while the other two solutions present a discontinuity at the origin and are either of the form 
\begin{equation}
\label{eq:u1omega}
u_{1,\omega}(x)=
    \begin{cases}
    \phi_\omega(x+x_+),\quad x>0\\
        \phi_\omega(x+x_-),\quad x<0,
    \end{cases}
\end{equation}
with $x_+>x_->0$, or of the form $u_{2,\omega}(x)=u_{1,\omega}(-x)$. In particular, $t_\pm:=\tanh\left(\frac{p-2}{2} \sqrt{\omega}x_\pm\right)$ solve the system \eqref{eq:t+t-pos}, namely 
\begin{equation}
\begin{cases}
        t_+(1-t_+^2)^{\frac{1}{p-2}}=t_-(1-t_-^2)^{\frac{1}{p-2}},\\
        t_+^{-1}-t_-^{-1}=-\beta \sqrt{\omega}.
    \end{cases}
\end{equation} 
It holds also that $0<t_-<\sqrt{\frac{p-2}{p}} < t_+ < 1$.
\end{lemma}
\begin{proof}
By the first equation in \eqref{eq:stationary}, any positive solution belonging to $H^1(\R^-)\oplus H^1(\R^+)$ has the form \eqref{eq:u1omega}, with $x_\pm$ proper translations to be determined.

Let us observe also that the only positive solution of \eqref{eq:stationary} with $u'(0^+)=0$ is given by the soliton centered at $0$. We are thus reduced to finding positive solutions of \eqref{eq:stationary} with $u'(0^+)\neq 0$. 

Denoting $t_\pm:=\tanh\left(\frac{p-2}{2} \sqrt{\omega}x_\pm\right)\in (-1,1)$, since
\begin{equation*}
    \phi'_\omega(x)=-\sqrt{\omega}\tanh\left(\frac{p-2}{2}\sqrt{\omega}x\right)\phi_\omega(x)
\end{equation*}
and $\phi_\omega(x_\pm)$ and $t_\pm$ do not vanish, then boundary condition in \eqref{eq:stationary} reduce to the system \eqref{eq:t+t-pos}.

Let us now notice that, if $u$ is a positive solution, then $u(-\cdot)$ is a positive solution too, hence without loss of generality we can focus on positive solutions satisfying $u'(0^+)< 0$, that correspond to $t_+,t_-\in(0,1)$. We can rewrite the second equation in \eqref{eq:t+t-pos} as
\begin{equation}\label{eq:Topo-Gigio-Va-In-Pensione}
	t_- h(t_+):=\frac{t_+}{\beta \sqrt{\omega} \, t_+ +1},
\end{equation}
from which it immediately follows that $0<t_-<t_+<1$ and $0<x_-<x_+$. Moreover, it holds that 
\begin{equation*}
 \lim_{t\to 0^+} h(t)=0,\quad \lim_{t\to 1} h(t)=\frac{1}{\beta\sqrt{\omega}+1}   
\end{equation*}
and $h$ is strictly increasing.

Let us now focus on the first equation in \eqref{eq:t+t-pos}, that can be rewritten as \begin{equation}\label{eq:SecondEqq}
t_+^{p-2}\left(1-t_+^2\right)=t_-^{p-2}\left(1-t_-^2\right).
\end{equation}
and let us note that $t_+=t_-$ is always a solution (but it is not compatible with the first equation). In order to discuss the existence of solutions $t_+\neq t_-$, let us define the function $f(t):=t^{p-2}(1-t^2)$ on the interval $(0,1)$: it is easy to check that 
\begin{equation*}
\lim_{t\to 0^+} f(t)=\lim_{t\to 1^-} f(t)=0
\end{equation*}
and $f$ is strictly increasing in $\left(0,\sqrt{\frac{p-2}{p}}\right]$ and strictly decreasing in $\left[\sqrt{\frac{p-2}{p}},1\right)$, hence for every $k\in\left(0,f\left(\sqrt{\frac{p-2}{p}}\right)\right)$ there exist exactly two solutions $0<t_1<\sqrt{\frac{p-2}{p}}<t_2<1$ such that $f(t_1)=f(t_2)=k$. This entails that for every  $t_+\in \left(0,\sqrt{\frac{p-2}{p}}\right)$ there exists a unique $t_-\neq t_+$, in particular $t_-\in\left(\sqrt{\frac{p-2}{p}},1\right)$, such that \eqref{eq:SecondEqq} is satisfied, and for every $t_+\in\left(\sqrt{\frac{p-2}{p}},1\right)$ there exists a unique $t_-\neq t_+$, in particular $t_-\in \left(0,\sqrt{\frac{p-2}{p}}\right)$, such that \eqref{eq:SecondEqq} is satisfied. Therefore, one can see $t_-$ as a function of $t_+$ in the interval $\left(0,1\right)\setminus\left\{\sqrt{\frac{p-2}{p}}\right\}$, namely $t_-=:g(t_+)$.  By using the implicit definition of $g$ given by 
\begin{equation*}
g(t)^{p-2}(1-g(t)^2)=t^{p-2}(1-t^2),
\end{equation*}
it turns out that $g$ can be extended by continuity to the whole interval $(0,1)$, satisfies the following limits 
\begin{equation*}
\lim_{t\to 0^+}g(t)=1,\quad  \lim_{t\to 1^-}g(t)=0
\end{equation*}
and, by the monotonicity properties of $f$, its derivative is given by
\begin{equation*}
g'(t)=\frac{(p-2)t^{p-3}-p t^{p-1}}{(p-2)g(t)^{p-3}-p g(t)^{p-1}}<0 \quad \text{for every}\quad t\in \left(0,1\right)\setminus\left\{\sqrt{\frac{p-2}{p}}\right\},
\end{equation*}
entailing that $g$ is strictly decreasing in the interval $\left(0,1\right)$. 

Therefore, the graphs of the functions $t_+\mapsto g(t_+)$ and $t_+\mapsto h(t_+)$ intersect each other below the line $t_-=t_+$, as displayed in Fig.~\ref{fig:solpos}, hence for any $\beta,\omega>0$, the system \eqref{eq:t+t-pos} admits a unique solution $(t_-,t_+)$, with $0<t_-<
\sqrt{\frac{p-2}{p}}<t_+<1$: in particular, $t_-\to \sqrt{\frac{p-2}{p}}$ and $t_+\to \sqrt{\frac{p-2}{p}}$ as $\omega\to 0$ and $t_-\to 0$ and $t_+\to 1$ as $\omega\to+\infty$. This concludes the proof.

{\color{red}

%
%
}

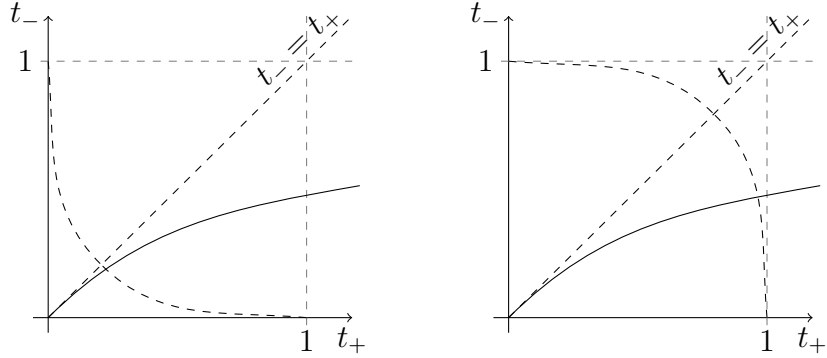
\begin{figure}[!ht]
	\begin{tikzpicture}
		\draw[->] (-0.2,0) -- (4,0);
		\draw[->] (0,-0.2) -- (0,4);
		\node at (-0.3,4) {$t_-$};
		\node at (4,-0.3) {$t_+$};
		\node at (-0.3,3.4) {$1$};
		\node at (3.4,-0.3) {$1$};
		
		\draw[-,dashed,color=gray] (3.4,-0.1) -- (3.4,4);
		\draw[-,dashed,color=gray] (-0.1,3.4) -- (4,3.4);
		
		\draw[-,color=black,dashed] (0,0) -- (4,4);
		\node[rotate=45] at (3.3,3.6) {\color{black}$t_-=t_+$};
            \draw[-,color=black] (0,0) to[in=190,out=45] (4.1,1.75);        
		
		
		\draw[-,color=black,dashed] (3.4,0) to[in=-45,out=175] (0.7,0.7);
		\draw[-,color=black,dashed] (0,3.4) to[in=90+45,out=-85] (0.7,0.7);
	\end{tikzpicture}
	$\qquad$
	\begin{tikzpicture}
		\draw[->] (-0.2,0) -- (4,0);
		\draw[->] (0,-0.2) -- (0,4);
		\node at (-0.3,4) {$t_-$};
		\node at (4,-0.3) {$t_+$};
		\node at (-0.3,3.4) {$1$};
		\node at (3.4,-0.3) {$1$};
		
		\draw[-,dashed,color=gray] (3.4,-0.1) -- (3.4,4);
		\draw[-,dashed,color=gray] (-0.1,3.4) -- (4,3.4);
		
		\draw[-,color=black,dashed] (0,0) -- (4,4);
		\node[rotate=45] at (3.3,3.6) {\color{black}$t_-=t_+$};
            \draw[-,color=black] (0,0) to[in=190,out=45] (4.1,1.75);          
		
		
		\draw[-,color=black,dashed] (3.4,0) to[in=-45,out=95] (2.7,2.7);
		\draw[-,color=black,dashed] (0,3.4) to[in=90+45,out=-5] (2.7,2.7);
	\end{tikzpicture}
	
	\caption{Graphical representation of the solutions of the system \eqref{eq:t+t-pos}. On the left, the case $p<3$; on the right, the case $p>3$. The solid and the dashed black lines represent the graph of the two functions of the system \eqref{eq:t+t-pos}.}\label{fig:solpos}
\end{figure}

\end{proof}

\begin{lemma}
\label{lem:mass-pos}
Let $u_{1,\omega}$ be the stationary state provided by Lemma \ref{lem:sol-pos}. Then the function 
\begin{equation*}
\omega\mapsto\|u_{1,\omega}\|_{L^2(\R)}^2
\end{equation*}
is strictly increasing and maps the interval $(0,+\infty)$ into $(0,+\infty)$.
\end{lemma}
\begin{proof}
Let us start by observing that 
\begin{equation}
\label{eq:l2norm-u1}
\|u_{1,\omega}\|_{L^2(\R)}^2=C_p\omega^{\frac{6-p}{2(p-2)}}\left\{\int_{-1}^{t_-} (1-t^2)^{\frac{4-p}{p-2}}\,dt +\int_{t_+}^{1} (1-t^2)^{\frac{4-p}{p-2}}\,dt\right\},
\end{equation}
with $C_p:=\frac{p^{\frac{2}{p-2}}}{(p-2)2^{\frac{4-p}{p-2}}}$ and  $t_+=t_+(\omega)>t_-=t_-(\omega)>0$ solving the system \eqref{eq:t+t-pos}. In particular, denoted with $f(\omega):=C_p\omega^{\frac{6-p}{2(p-2)}}$ and $g(\omega):=\int_{-1}^{t_-} (1-t^2)^{\frac{4-p}{p-2}}\,dt +\int_{t_+}^{1} (1-t^2)^{\frac{4-p}{p-2}}\,dt$, one can write $\|u_{1,\omega}\|_{L^2(\R)}^2=f(\omega)g(\omega)$. First, note that $f$ is strictly increasing in $(0,+\infty)$. On the other hand, let us observe that
\begin{equation}
\label{eq:g'omega}
g'(\omega)=(1-t_-^2(\omega))^{\frac{4-p}{p-2}}t_-'(\omega)-(1-t_+^2(\omega))^{\frac{4-p}{p-2}}t_+'(\omega). 
\end{equation}
In order to establish the sign of $g'(\omega)$, let us note that by the Implicit Function Theorem applied to \eqref{eq:t+t-pos} it follows that
\begin{equation*}
\begin{cases}
t_+'=\frac{(1-t_-^2)^{\frac{1}{p-2}-1}(1-\frac{p}{p-2}t_-^2)}{(1-t_+^2)^{\frac{1}{p-2}-1}(1-\frac{p}{p-2}t_+^2)}t_-',\\
\frac{t_+'}{t_+^2}-\frac{t_-'}{t_-^2}=-\frac{\beta}{2\sqrt{\omega}}.
\end{cases}
\end{equation*}
From the first equation, since $0<t_-<\sqrt{\frac{p-2}{p}}<t_+<1$, we can note that $t_+'(\omega)$ and $t_-'(\omega)$ have opposite sign. 
From the second equation, one gets
\begin{equation*}
    t_+'=\frac{t_+^2}{t_-^2}t_-'-\frac{\beta t_+^2}{2\sqrt{\omega}},
\end{equation*}
that, substituted in the first equation, leads to 
\begin{equation*}
\left(\frac{(1-t_-^2)^{\frac{1}{p-2}-1}(1-\frac{p}{p-2}t_-^2)t_-^2}{(1-t_+^2)^{\frac{1}{p-2}-1}(1-\frac{p}{p-2}t_+^2)t_+^2}-1\right)t_-'=-\frac{\beta t_-^2}{2\sqrt{\omega}}.
    \end{equation*}
Since both the right-hand side and the factor multiplying $t_-'$ are negative, it follows that $t_-'(\omega)>0$, hence $t'_+(\omega)<0$. By \eqref{eq:g'omega}, it follows that $g'(\omega)>0$. Being both $f$ and $g$ strictly increasing, also the function $\omega\mapsto\|u_{1,\omega}\|_{L^2(\R)}^2$ is strictly increasing. Moreover, it is straightforward to prove that
\begin{equation*}
\lim_{\omega\to 0^+}\|u_{1,\omega}\|_{L^2(\R)}^2=0,\quad \lim_{\omega\to+\infty}\|u_{1,\omega}\|_{L^2(\R)}^2=+\infty,
\end{equation*}
concluding the proof.
\end{proof}

In the next lemma, we investigate the second class of solutions, namely solutions of the system \eqref{eq:stationary} that are negative on $\R^-$ and positive on $\R^+$.

\begin{lemma}
\label{lem:stat-negpos}
	Let $\beta>0$. If $\omega>0$. Then any solution of the system \eqref{eq:stationary} belonging to $H^1(\R^-)\oplus H^1(\R^+)$ and being negative on $\R^-$ and positive on $\R^+$ is of the form
	\begin{equation}
		\label{eq:u-negpos}u(x)=\left\{\begin{array}{lcl}
			\phi_\omega(x-x_+) \, , &\quad & x>0 \\
			-\phi_\omega(x+x_-) \, , & & x<0
		\end{array}\right.
	\end{equation}
	where $x_\pm>0$ and $t_\pm:=\tanh(\frac{p-2}{2} \sqrt{\omega} x_\pm)$ solve the system
	\begin{equation}\label{eq:t+t-posneg}
		\begin{cases}
			t_+(1-t_+^2)^{\frac{1}{p-2}} = t_-(1-t_-^2)^{\frac{1}{p-2}} \\
			t_+^{-1}+t_-^{-1}=\beta \sqrt{\omega} \, .
		\end{cases}
	\end{equation}
	In particular
	\begin{itemize}
		\item[(a)] there are no stationary solutions  for $\omega \leq \frac{4}{\beta^2}$
		\item[(b)] there is a unique odd stationary solution for $\omega \in \left(\frac{4}{\beta^2},\frac{p}{p-2} \frac{4}{\beta^2}\right]$
		\item[(c)] there exist one odd stationary solution and two non-odd stationary solutions for $\omega \in \left(\frac{p}{p-2} \frac{4}{\beta^2},+\infty\right)$.
	\end{itemize} 
\end{lemma}
\begin{proof}
	Let us observe that, as in Lemma \ref{lem:sol-pos}, $u$ must be a proper translation of the soliton $\phi_\omega$ on each half-line. Since we look for solutions with $u(0^+)>0$ and $u(0^-)<0$, it follows that $u'(0^+)=u'(0^-)>0$, entailing that $u$ is of the form \eqref{eq:u-negpos}, with $x_\pm>0$. Denoting with $t_\pm=\tanh\left(\frac{p-2}{2}\sqrt{\omega}x_\pm\right)$ and arguing similarly as in Lemma \ref{lem:sol-pos}, the system \eqref{eq:t+t-posneg} follows. In order to conclude the proof, we observe that system \eqref{eq:t+t-posneg} is exactly the same system as in \cite[Equation (5.2)]{Adami-Noja-CMP2013-1D-NLS-deltaprime}. In particular, in the proof of \cite[Theorem 5.3]{Adami-Noja-CMP2013-1D-NLS-deltaprime} the authors proved that:
    \begin{itemize}
        \item if $\omega\leq \frac{4}{\beta^2}$ then the system \eqref{eq:t+t-posneg} admits no  solution;
        \item if $\frac{4}{\beta^2}<\omega\leq \frac{p}{p-2}\frac{4}{\beta^2}$, then there exists only one solution, and it is given by $t_-=t_+=\frac{2}{\beta\sqrt{\omega}}$;
        \item if $\omega>\frac{p}{p-2}\frac{4}{\beta^2}$, then there exists the solution $t_-=t_+=\frac{2}{\beta\sqrt{\omega}}$ and two other solutions with $t_-\neq t_+$, given respectively by $(t_1,t_2)$ and $(t_2,t_1)$. Points $(a), (b)$ and $(c)$ concerning the number of stationary solutions at fixed $\omega$ immediately follow from the multiplicity of the solutions to the system \eqref{eq:t+t-posneg}.  
    \end{itemize}

\begin{figure}[h!]
\label{fig:sol-posneg}
	\begin{tikzpicture}
		\draw[->] (-0.2,0) -- (4,0);
		\draw[->] (0,-0.2) -- (0,4);
		\node at (-0.3,4) {$t_-$};
		\node at (4,-0.3) {$t_+$};
		\node at (-0.3,3.4) {$1$};
		\node at (3.4,-0.3) {$1$};
		
		\draw[-,dashed,color=gray] (3.4,-0.1) -- (3.4,4);
		\draw[-,dashed,color=gray] (-0.1,3.4) -- (4,3.4);
		
		\draw[-,color=black,dashed] (0,0) -- (3.4,3.4);
		
		\draw[-,color=black,dashed] (3.4,0) to[in=-45,out=95] (2.2,2.2);
		\draw[-,color=black,dashed] (0,3.4) to[in=90+45,out=-5] (2.2,2.2);
		
		\draw[-,color=black] (0.4,4) to[in=178,out=-88] (4,0.4);
		
		\draw[-,color=black] (1.365,4) to[in=174,out=-84] (4,1.365);
		
		\draw[-,color=black] (2.3,4) to[in=172,out=-82] (4,2.3);
		
		\node at (5.4,0.4) {\color{black}$\omega> \frac{p}{p-2}\frac{4}{\beta^2}$};
		
		\node at (5.4,1.365) {\color{black}$\omega=\frac{p}{p-2}\frac{4}{\beta^2}$};
		
		\node at (5.4,2.3) {\color{black}$\frac{4}{\beta^2}<\omega <\frac{p}{p-2}\frac{4}{\beta^2} $};
	\end{tikzpicture}
	\caption{Graphical representation of the solutions to the system \eqref{eq:t+t-posneg}. The black dashed corresponds to the set of solutions of the first equation. Solid black lines correspond to the solutions of the second equation at different values of $\omega$. For $\omega\leq \frac{4}{\beta^2}$ the hyperbola does not have points in the square.}
\end{figure}
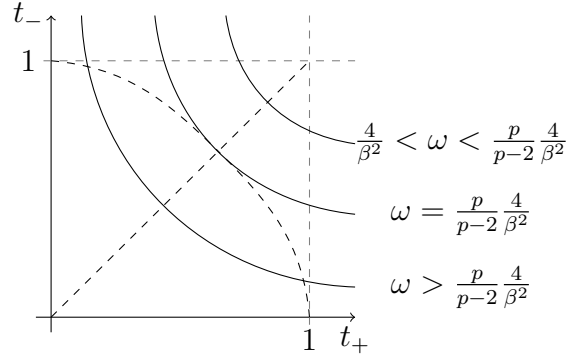
\end{proof}

\section{Action minimizers: proof of Theorem \ref{thm:ex-AM}} \label{sec:action}

This section is devoted to the proof of Theorem \ref{thm:ex-AM} about the existence, the uniqueness and the identification of the specific profile of action minimizers at frequency $\omega$.

First of all, let us introduce the auxiliary functional $\widetilde{S}:H^1(\R^-)\oplus H^1(\R^+)\to \R$, defined as
\begin{equation}
\label{eq:S-tilde}
    \widetilde{S}(u):=\frac{p-2}{2p}\|u\|_{L^p(\R)}^p,
\end{equation}
and observe that
\begin{equation}
\label{eq_S-om-S-tilde}
S_{\beta,\omega}(u)=\frac{1}{2}I_{\beta,\omega}(u)+\widetilde{S}(u)=\frac{1}{p}I_{\beta,\omega}(u)+\frac{p-2}{2p}\left(Q_{\beta}(u)+\omega\|u\|_{L^2(\R)}^2\right). 
\end{equation}
In particular, this 
entails that $S_{\beta,\omega} (v)=\widetilde{S}(v)$ for every $v\in N_{\beta,\omega}$ and 
\begin{equation*}
    d_\beta(\omega)=\inf_{v\in N_{\beta,\omega}} \widetilde{S}(v).
\end{equation*}
Moreover, in the following lemma we state a stronger result concerning the relation between the minimization problems involving $S_{\beta,\omega}$ and $\widetilde{S}$: since its proof is standard, we refer to \cite[Lemma 4.5]{ABcpde} and we do not report it here.

\begin{lemma}
\label{lem:equiv-normap}
Let $\beta\in \R \setminus \{0\}$ and $\omega>0$. Then $u\in H^1(\R^-)\oplus H^1(\R^+)\setminus\{0\}$ is an action minimizer at frequency $\omega$ if and only if  $I_{\beta,\omega}(u)\leq 0$ and 
\begin{equation*}
\widetilde{S}(u)=\inf_{v\in \widetilde{N}_{\beta,\omega}}\widetilde{S}(v),
\end{equation*}
where 
\begin{equation*}
\widetilde{N}_{\beta,\omega}:=\left\{v\in H^1(\R^-)\oplus H^1(\R^+)\,, v\neq 0\,,\,I_{\beta,\omega}(v)\leq 0\right\} \, .
\end{equation*}
Furthermore, it holds that
\begin{equation*}
   d_\beta(\omega)= \inf_{v\in \widetilde{N}_{\beta,\omega}}\widetilde{S}(v).
\end{equation*}
\end{lemma}

In the next proposition, we collect some well-known results (\cite[Lemma 2.4]{DST-action-energy} and \cite[Theorem 8.1.6]{cazenave-03}) about action minimizers at frequency $\omega$ in the case $\beta=0$, that corresponds to the absence of the point interactions at the origin.

\begin{proposition}\label{prop:SiFossiFoco}
It holds that
\begin{equation}
    d_{0}(\omega)
    \begin{cases}
       =0,\quad &\text{if}\quad\omega \leq0,\\
       >0,\quad &\text{if}\quad\omega>0,
    \end{cases}
\end{equation}
and action minimizers at frequency $\omega$ exist if and only if $\omega> 0$. Moreover, up to translations and multiplication by a phase factor, it is unique and given by the soliton $\phi_\omega$. 
\end{proposition}

We are now ready for the proof of Theorem \ref{thm:ex-AM}

\begin{proof}[Proof of Theorem \ref{thm:ex-AM}]
Since $H^1(\R) \subset H^1(\R^-) \oplus H^1(\R^+)$, then $d_{\beta}(\omega) \leq d_{0}(\omega)$ and, therefore, $0\leq d_{\beta}(\omega)\leq d_{0}(\omega)$. If $\omega<0$, by Proposition \ref{prop:SiFossiFoco} $d_{0}(\omega)=0$, therefore $d_{\beta}(\omega)=0$ and the infimum is never attained. 

Fix now $\omega>0$ and let $(u_n)_n$ be a minimizing sequence for $\widetilde{S}$ in $\widetilde{N}_{\beta,\omega}$, i.e.~$I_{\beta,\omega}(u_n)\leq0$ and $\widetilde{S}(u_n)\to d_\beta(\omega)$ as $n\to+\infty$. We divide the proof in four steps. 

\textit{Step 1. Weak convergence of $u_n$ to $u$.} First of all, $(u_n)_n$ is bounded in $L^p(\R)$ as it converges. Then, since $I_{\beta,\omega}(u_n)\leq 0$, it follows that $(u_n)_n$ is bounded also in $H^1(\R^-)\oplus H^1(\R^+)$. Therefore, there exists $u\in H^1(\R^-)\oplus H^1(\R^+)$ such that, up to subsequences, $u_n\deb u$ weakly in $H^1(\R^-)\oplus H^1(\R^+)$ and in $L^p(\R)$ and $u_n\to u$ strongly in $L^r_{\text{loc}}(\R)$ for every $r\in[2,+\infty]$, entailing that $|u_n(0^+)-u_n(0^-)|^2\to |u(0^+)-u(0^-)|^2$ as $n\to+\infty$.

\textit{Step 2. There results that $0<d_\beta(\omega)<d_0(\omega)$.} Let us consider  $v\in \widetilde{N}_{\beta,\omega}$. By Sobolev inequality, there exists $C>0$ depending only on $p$ such that
\begin{equation*}
    \begin{split}
        0\geq I_{\beta,\omega}(v)&> \|v'\|_{L^2(\R^-)}^2+\|v'\|_{L^2(\R^+)}^2-\|v\|_{L^p(\R)}^p \\
        &\geq C\|v\|_{L^p(\R^-)}^2+C\|v\|_{L^p(\R^+)}^2-\|v\|_{L^p(\R)}^p\geq \frac{C}{2}\|v\|_{L^p(\R)}^2-\|v\|_{L^p(\R)}^p.
    \end{split}
\end{equation*}
Therefore, $\|v\|_{L^p(\R)}^{p-2}\geq \frac{C}{2}$ (with $C$ depending only on $p$) and, by Lemma \ref{lem:equiv-normap}, it follows that  
\begin{equation*}
d_{\beta}(\omega)=\inf_{v\in \widetilde{N}_{\beta,\omega}} \frac{p-2}{2p}\|v\|_{L^p(\R)}^p\geq \frac{p-2}{2p}\left(\frac{C}{2}\right)^{\frac{p}{p-2}}>0.
\end{equation*}
Moreover, let us consider the stationary state $u_{1,\omega}$ in \eqref{eq:u1omega}.  Since $x_- < x_+$, one has 
\begin{equation*}
\begin{split}
d_\beta(\omega)\leq  & \frac{p-2}{2p}\|u_{1,\omega}\|_{L^p(\R)}^p=\frac{p-2}{2p}\left\{\int_{-\infty}^{x_-}|\phi_\omega|^p\,dx+\int_{x_+}^{+\infty}|\phi_\omega|^p\,dx\right\} \\ < & \frac{p-2}{2p}\|\phi_\omega\|_{L^p(\R)}^p=d_0(\omega),
\end{split}
\end{equation*}
concluding the proof of Step 2.

\textit{Step 3. It holds that $I_{\beta,\omega}(u_n)\to 0$}. Suppose by contradiction that $I_{\beta,\omega}(u_n)\not\to 0$, i.e.~up to subsequences $I_{\beta,\omega}(u_n)\to -M<0$, and define the sequence $v_n:=\alpha_n u_n$, with
\begin{equation*}
\alpha_n:=\frac{\left(Q_\beta(u_n)+\omega\|u_n\|_{L^2(\R)}^2\right)^{\frac{1}{p-2}}}{\|u_n\|_{L^p(\R)}^{\frac{p}{p-2}}}<1.
\end{equation*} 
Since 
\begin{equation*}
    \lim_{n\to+\infty} \alpha_n=\lim_{n\to+\infty} \left(1+\frac{I_{\beta,\omega}(u_n)}{\|u_n\|_{L^p(\R)}^p}\right)^{\frac{1}{p-2}}=\left(1-\frac{(p-2)M}{2p  d_\beta(\omega)}\right)^{\frac{1}{p-2}}<1,
\end{equation*}
we get
\begin{equation*}
    \lim_{n\to+\infty} \widetilde{S}(v_n)=\lim_{n\to+\infty} \alpha_n^p\widetilde{S}(u_n) = \big(\lim_{n \to \infty} \alpha_n^p\big) \big(\lim_{n \to \infty} \widetilde{S}(u_n)\big) < \lim_{n\to +\infty} \widetilde{S}(u_n)
\end{equation*}

that, together with $I_{\beta,\omega}(u_n)=0$, contradicts th assumption that $(u_n)_n$ is a minimizing sequence for $\widetilde{S}$ in $\widetilde{N}_{\beta,\omega}$, and so $\lim_{n\to+\infty}I_{\beta,\omega}(u_n)=0$.

\textit{Step 4. It holds that $u\in \widetilde{N}_{\beta,\omega}$.} Let us start proving that $u\not \equiv 0$. Suppose by contradiction $u\equiv 0$ and define the sequence $w_n:=\gamma_n u_n$, with
\begin{equation*}
\gamma_n:=\frac{\left(Q_0(u_n)+\omega\|u_n\|_{L^2(\R)}^2\right)^{\frac{1}{p-2}}}{\|u_n\|_{L^p(\R)}^{\frac{p}{p-2}}}. 
\end{equation*}
By $I_{\beta,\omega}(u_n)\to 0$ and $|u_n(0^+)-u_n(0^-)|^2\to |u(0^+)-u(0^-)|^2=0$, it follows that
\begin{equation*}
\lim_{n\to+\infty}\gamma_n=\lim_{n\to+\infty}\left(1-\frac{\frac{1}{\beta}|u_n(0^+)-u_n(0^-)|^2}{\|u_n\|_{L^p(\R)}^p}\right)^{\frac{1}{p-2}}=1.
\end{equation*}
Therefore, since definitively $I_{0,\omega}(w_n)\leq 0$, then 
\begin{equation*}
d_0(\omega)\leq \lim_{n\to+\infty}\widetilde{S}(w_n)=\lim_{n\to+\infty} \gamma_n\widetilde{S}(u_n)=d_\beta(\omega),
\end{equation*}
contradicting Step 2, then $u\not\equiv 0$. 

We are thus left to prove $I_{\beta,\omega}(u)\leq 0$. Assume again by contradiction that $I_{\beta,\omega}(u)>0$. Since $(u_n)_n$ is bounded in $L^p(\R)$ and $u_n\to u$ a.e. in $\R$, by Brezis-Lieb lemma \cite{Brezis-Lieb-83} one has
\begin{equation}
\label{eq:Stilde-BL}
    \widetilde{S}(u_n)-\widetilde{S}(u_n-u)-\widetilde{S}(u)\to 0,\quad n\to+\infty.
\end{equation}
Moreover, since $u_n\deb u$ weakly in $H^1(\R^-)\oplus H^1(\R^+)$ and $|u_n(0^+)-u_n(0^-)|^2\to |u(0^+)-u(0^-)|^2$, we have 
\begin{equation}
\label{eq:I-BL}
    I_{\beta,\omega}(u_n)-I_{\beta,\omega}(u_n-u)-I_{\beta,\omega}(u)\to 0,\quad n\to+\infty,
\end{equation}
hence by Step 3 and \eqref{eq:I-BL} it follows that $I_{\beta,\omega}(u_n-u)\to -I_{\beta,\omega}(u)<0$ and in particular, for $n$ large enough, we have $u_n-u \in \widetilde{N}_{\beta,\omega}$. Therefore, by \eqref{eq:Stilde-BL} 
\begin{equation*}
d_\beta(\omega)\leq \lim_{n\to+\infty}\widetilde{S}(u_n-u)=\lim_{n\to+\infty}\left\{\widetilde{S}(u_n)-\widetilde{S}(u)\right\}=d_\beta(\omega)-\widetilde{S}(u)<d_\beta(\omega),
\end{equation*}
which is a contradiction and concludes the proof of Step 4.

\textit{Conclusion.} Since $u_n\deb u$ weakly in $L^p(\R)$, by lower semicontinuity
\begin{equation*}
    \widetilde{S}(u)\leq \liminf_{n\to+\infty}\widetilde{S}(u_n)=d_\beta(\omega),
\end{equation*}
that, together with $I_{\beta,\omega}(u)\leq 0$, entails that $u$ is an action minimizer at frequency $\omega$. Moreover, it follows that $u_n\to u$ strongly in $L^p(\R)$ and, by Lemma \ref{lem:equiv-normap}, $I_{\beta,\omega}(u)=0$ and $u_n\to u$ strongly in $H^1(\R^-)\oplus H^1(\R^+)$. 

Finally, among all the stationary states, classified in Lemma \ref{lem:sol-pos} and Lemma \ref{lem:stat-negpos}, one aims at finding which of these are the action minimizers at frequency $\omega>0$. In particular, repeating the computations in Step 2, it is easy to check that
\begin{equation*}
\widetilde{S}(u_{1,\omega})=\widetilde{S}(u_{2,\omega})<\widetilde{S}(\phi_\omega),
\end{equation*}
where $\phi_\omega$ is the soliton at frequency $\omega$ and $u_{1,\omega}$, $u_{2,\omega}$ are the two other positive stationary solutions found in Lemma \ref{lem:sol-pos}. On the contrary, if one takes stationary states $v$ that change sign, as in Lemma \ref{lem:stat-negpos}, it is easy to show that
\begin{equation*}
\widetilde{S}(v)=\frac{p-2}{2p}\left\{\int_{-\infty}^{x_-}|\phi_\omega|\,dx+\int_{-x_+}^{+\infty}|\phi_\omega|\,dx\right\}>\widetilde{S}(\phi_\omega),
\end{equation*}
where the strict inequality follows from $x_\pm>0$.

Since all the stationary solutions are given in Lemma \ref{lem:sol-pos} and \ref{lem:stat-negpos} (up to a multiplication by a phase factor), there results that the only action minimizers are given by $e^{i\theta}u_{1,\omega}$ and $e^{i\theta}u_{2,\omega}$, with $\theta\in \R$, thus up to the additional symmetry $u\mapsto u(-\cdot)$ the unique action minimizer is given by $u_{1,\omega}$.
\end{proof}

\section{Ground states at fixed mass: proof of Theorem \ref{thm:ex-GS}} \label{sec:minchiazzio}

This section is devoted to the proof of Theorem \ref{thm:ex-GS} about the existence and uniqueness of ground states at fixed mass.

A crucial role in the identification of the specific profiles of normalized ground states is played by the next lemma, that establishes a bridge between ground states at fixed mass and action minimizers at fixed frequency. Even though an analogous result has been proven in \cite{DST-action-energy,Jeanjean-22} for the standard NLS energy, we report here the proof for the sake of completeness.

\begin{lemma}
\label{lem:link-GS-AM}
Let $u$ be a ground state at mass $\mu$. Then  $u$ is an action minimizer at frequency \begin{equation*}
\omega:=\mu^{-1}\left(\|u\|_{L^p(\R)}^p-Q_\beta(u)\right).
\end{equation*}
\end{lemma}

\begin{proof}
Let $u$ be a ground state at mass $\mu$ and $\omega=\mu^{-1}\left(\|u\|_{L^p(\R)}^p-Q_\beta(u)\right)$ be the associated Lagrange multiplier. Suppose by contradiction that $u$ is not an action minimizer at frequency $\omega$, that is there exists $f\in N_{\beta,\omega}$ such that $S_{\beta,\omega}(f)<S_{\beta,\omega}(u)$. If $\theta>0$ is such that $\|\theta f\|_{L^2(\R)}^2=\mu$, then 
\begin{equation*}
S_{\beta,\omega}(\theta f)=\frac{\theta^2}{2}\left(Q_\beta(f)+\omega\|f\|_{L^2(\R)}^2\right)-\frac{\theta^p}{p}\|f\|_{L^p(\R)}^p.
\end{equation*}
Computing the derivative with respect to $t$ of the function $S_{\beta,\omega}(tf)$ and using $f\in N_{\beta,\omega}$, there results that $t=1$ is a global maximum for $S_{\beta,\omega}(\cdot f)$ in $(0,+\infty)$, namely $S_{\beta,\omega}(t f)\leq S_{\beta,\omega}(f)$ for every $t>0$, in particular $S_{\beta,\omega}(\theta f)\leq S_{\beta,\omega}(f)$. As a consequence
\begin{equation*}
E_\beta(\theta f)=S_{\beta,\omega}(\theta f)-\frac{\omega}{2}\|\theta f\|_{L^2(\R)}^2\leq S_{\beta,\omega}(f)-\frac{\omega}{2}\|u\|_{L^2(\R)}^2<S_{\beta,\omega}(u)-\frac{\omega}{2}\|u\|_{L^2(\R)}^2=E_\beta(u),
\end{equation*}
which contradicts the fact that $u$ is a ground state at mass $\mu$, entailing the thesis.
\end{proof}
\begin{proof}[Proof of Theorem \ref{thm:ex-GS}]
    Let us start proving the existence of ground states at mass $\mu$. First of all, we observe that $\Eps_\beta(\mu)\leq E_\beta(\phi_\mu)=E_0(\phi_\mu)$, thus two cases are possible: either $\Eps_\beta(\mu)=E_0(\phi_\mu)$ or $\Eps_\beta(\mu)<E_0(\phi_\mu)$. Since in the first case we can conclude that $\phi_\mu$ itself is a ground state at mass $\mu$, we focus on the second case, namely $\Eps_\beta(\mu)<E_0(\phi_\mu)$.
    
    Let $(u_n)_n$ be a minimizing sequence for $E_\beta$ at mass $\mu$, namely $E_\beta(u_n)\to\Eps_\beta(\mu)$ as $n\to+\infty$ and $\|u_n\|_{L^2(\R)}^2=\mu$. Without loss of generality, we can assume that $u_n\geq0$ for every $n\in \N$ and $u_n(0^+)\geq u_n(0^-)$: indeed, if $u_n$ is not positive, then one can pass to $|u_n|$ and the energy $E_\beta$ does not increase, while if $u_n(0^-)>u_n(0^+)$, then one can pass to $u_n(-\cdot)$ and the energy $E_\beta$ is preserved. 
    
    By $0>E_0(\phi_\mu)>\Eps_\beta(\mu)$ and Gagliardo-Nirenberg inequality 
    \begin{equation*}
\|v\|_{L^p(\R)}^p\leq C_p\left(\|v'\|_{L^2(\R^-)}^2+\|v'\|_{L^2(\R^+)}^2\right)^{\frac{p}{4}-\frac{1}{2}}\|v\|_{L^2(\R)}^{\frac{p}{2}+1}\quad \forall\, v\in H^1(\R^-)\oplus H^1(\R^+),
    \end{equation*} there results that for sufficiently large $n$
    \begin{equation*}
    0>E_\beta(u_n)\geq \frac{1}{2}\left(\|u_n'\|_{L^2(\R^-)}^2+\|u_n'\|_{L^2(\R^+)}^2\right)-\frac{C}{p}\left(\|u_n'\|_{L^2(\R^-)}^2+\|u_n'\|_{L^2(\R^+)}^2\right)^{\frac{p}{4}-\frac{1}{2}}\mu^{\frac{p}{4}+\frac{1}{2}},
    \end{equation*}
    hence $(u_n)_n$ is bounded in $H^1(\R^-)\oplus H^1(\R^+)$ since $p<6$. This entails also that $(u_n)_n$ is bounded in $L^p(\R)$, thus there exists $u\in H^1(\R^-)\oplus H^1(\R^+)$ such that, up to subsequences, $u_n\deb u$ weakly in $H^1(\R^-)\oplus H^1(\R^+)$ as well as in $L^p(\R)$, and $u_n\to u$ strongly in $L^r_{\text{loc}}(\R)$ for every $r\geq 2$: in particular, $|u_n(0^+)-u_n(0^-)|^2\to |u(0^+)-u(0^-)|^2$.

    Let us now focus on the mass of the limit $u$, namely $m:=\|u\|_{L^2(\R)}^2$. By weak lower semicontinuity of the $L^2$ norm, we have that $m\leq \liminf_n \|u_n\|_{L^2(\R)}^2=\mu$.

    Let us start assuming that $m=0$, i.e.~$u\equiv 0$, that implies that $|u_n(0^+)-u_n(0^-)|^2\to 0$ as $n\to+\infty$. Let us define the sequence
    \begin{equation*}
    v_n=
    \begin{cases}
    u_n(x),\quad &x< 0\\
    u_n(0^-)+x, \quad &0\leq x \leq u_n(0^+)-u_n(0^-)\\ 
        u_n(x-(u_n(0^+)-u_n(0^-))),\quad &x> u_n(0^+)-u_n(0^-),\\
        \end{cases}
    \end{equation*}
    which belongs to $H^1(\R)$ and whose mass is given by 
        \begin{equation*}
        \|v_n\|_{L^2(\R)}^2=\mu+\frac{1}{3}u_n^3(0^+)-\frac{1}{3}u_n^3(0^-)=\mu+o(1),\quad \text{as} \quad n\to+\infty,
    \end{equation*}
    so that the sequence $w_n:=\frac{\sqrt{\mu}}{\|v_n\|_{L^2(\R)}}v_n$ belongs to $H^1_\mu(\R)$ and $E_0(w_n)-E_\beta(u_n)\to 0$. Therefore, 
    \begin{equation*}
\Eps_\beta(\mu)=\lim_{n\to+\infty} E_\beta(u_n)=\lim_{n\to+\infty}E_0(w_n)\geq E_0(\phi_\mu),
    \end{equation*}
    which is in contradiction with the assumption that $\Eps_\beta(\mu)<E_0(\phi_\mu)$, hence $m>0$.

    Assume now that $0<m<\mu$. Since $\frac{\mu}{\|u_n-u\|_{L^2(\R)}^2}\to\frac{\mu}{\mu-m}<1$, for $n$ sufficiently large
    \begin{equation*}
     \Eps_\beta(\mu)\leq E_\beta\left({\textstyle \sqrt{\frac{\mu}{\|u_n-u\|_{L^2(\R)}^2}}}(u_n-u)\right)<\frac{\mu}{\|u_n-u\|_{L^2(\R)}^2} E_\beta(u_n-u),
    \end{equation*}
    thus
    \begin{equation}
    \label{eq:E-un-u}
     \liminf_n E_\beta(u_n-u)\geq \frac{\mu-m}{\mu}\Eps_\beta(\mu).   
    \end{equation}
    Moreover, arguing similarly, one can obtain 
    \begin{equation}
    \label{eq:E-u}
        E_\beta(u)>\frac{m}{\mu}\Eps_\beta(\mu).
    \end{equation}

    By using Brezis-Lieb Lemma \cite{Brezis-Lieb-83} (since $u_n\deb u$ in $H^1(\R^-)\oplus H^1(\R^+)$ and $u_n\to u$ a.e. on $\R$), it follows that
    \begin{equation}
    \label{eq:E-BL}
        \lim_{n\to+\infty} E_\beta(u_n)=\lim_{n\to+\infty}E_\beta(u_n-u)+E_\beta(u).
    \end{equation}
    By combining \eqref{eq:E-un-u}, \eqref{eq:E-u} and \eqref{eq:E-BL}, we obtain
    \begin{equation*}
\Eps_\beta(\mu)=\lim_n E_\beta(u_n)=\lim_n E_\beta(u_n-u)+E_\beta(u)>\frac{\mu-m}{\mu}\Eps_\beta(\mu)+\frac{m}{\mu}\Eps_\beta(\mu)=\Eps_\beta(\mu),
    \end{equation*}
    which is again a contradiction, hence $m=\mu$. Therefore, since $u_n\to u$ in $L^2(\R)$ and $u_n$ is uniformly bounded in $L^\infty(\R)$, one gets $u_n\to u$ in $L^p(\R)$ and
    \begin{equation*}
        E_\beta(u)\leq \liminf_n E_\beta(u_n)=\Eps_\beta(\mu),
    \end{equation*}
    i.e.~$u$ is a ground state of $E_\beta$ at mass $\mu$. 

   Let us exclude now that $\Eps_\beta(\mu)=E_0(\phi_\mu)$. Indeed, if this is the case, then $\phi_\mu$ is a ground state for $E_\beta$ at mass $\mu$. By Lemma \ref{lem:link-GS-AM}, $u$ is also an action minimizer at frequency $\omega:=\mu^{-1}\left(\|u\|_{L^p(\R)}^p-Q_\beta(u)\right)>0$, but this is in contradiction with Theorem \ref{thm:ex-AM}: indeed, action minimizers at frequency $\omega$ never coincide with the soliton $\phi_{\omega(\mu)}$, where $\omega(\mu)$ is the only $\omega>0$ for which $\phi_{\omega}$ has mass $\mu$ (see \eqref{eq:omega-mu}). 
   Therefore, $\Eps_\beta(\mu)<E_0(\phi_\mu)$.

   Moreover, repeating the same argument for every stationary states except $u_{1,\omega}$ and $u_{2,\omega}$, one can exclude that they are ground states at mass $\mu$. Therefore, only $u_{1,\omega}$ or $u_{2,\omega}$ can be ground state at mass $\mu$. Since by Lemma \ref{lem:mass-pos} there exists a unique $\omega(\mu)$ such that $u_{1,\omega(\mu)}$ (and also $u_{2,\omega(\mu)}$) has mass $\mu$, then $u_{1,\omega(\mu)}$ is the only ground state at mass $\mu$, up to multiplication by a phase factor and the transformation $u\mapsto u(\cdot)$. This concludes the proof.

\end{proof}

\def\cprime{$'$}

\end{document}